\documentclass[letterpaper,10pt,conference]{ieeeconf}

\usepackage{amsmath,amsfonts,amssymb,graphicx,graphics,enumerate,dsfont,float,color,bm}
\usepackage{filecontents}
\IEEEoverridecommandlockouts
\def\DelX0{\Delta(X_0)}
\def\delpi0{\delta^{\pi}}

\newtheorem{Theorem}{Theorem}

\newtheorem{Lemma}{Lemma}
\newtheorem{Proposition}{Proposition}

\newcommand{\comment}[1]{}

\title{\LARGE \bf Performance Bounds for the $k$-Batch Greedy Strategy \\in Optimization Problems with Curvature}

\author{Yajing Liu, Zhenliang Zhang, Edwin K. P. Chong, and Ali Pezeshki
\thanks{This work is supported in part by NSF under award CCF-1422658, and by the CSU Information Science and Technology Center (ISTeC).}%
\thanks{Y. Liu is with the Department of Electrical and Computer Engineering, Colorado State University, Fort Collins, CO 80523, USA {\tt\small yajing.liu@ymail.com}}%
\thanks{Z. Zhang is with Qualcomm Flarion Technology, Bridgewater, NJ 08873, USA. He was with the Department of Electrical and Computer Engineering, Colorado State University, Fort Collins, CO 80523, USA {\tt\small zzl.csu@gmail.com}}
\thanks{E. K. P. Chong and A. Pezeshki are with the Department of Electrical and Computer Engineering, and the Department of Mathematics, Colorado State University, Fort Collins, CO 80523, USA {\tt\small Edwin.Chong,Ali.Pezeshki@Colostate.Edu}}%
}

\begin{document}

\maketitle
\thispagestyle{empty}
\pagestyle{empty}

\begin{abstract}
The $k$-batch greedy strategy is an approximate algorithm to solve optimization problems where the optimal solution is hard to obtain. Starting with the empty set, the $k$-batch greedy strategy adds a batch of $k$ elements to the current solution set with the largest gain in the objective function  while satisfying the constraints. In this paper, we bound the performance of the $k$-batch greedy strategy with respect to the optimal strategy by defining the total curvature $\alpha_k$. We show that when the objective function is nondecreasing and submodular, the $k$-batch greedy strategy satisfies a harmonic bound  $1/(1+\alpha_k)$ for a general matroid constraint and  an exponential bound $\left(1-(1-{\alpha}_k/{t})^t\right)/{\alpha}_k$ for a uniform matroid constraint, where $k$ divides the cardinality of the maximal set in the general matroid, $t=K/k$ is an integer, and $K$ is the rank of the uniform matroid.
We also compare the performance of the $k$-batch greedy strategy with that of the $k_1$-batch greedy strategy when $k_1$ divides $k$. Specifically, we prove that when the objective function is nondecreasing and submodular, the $k$-batch greedy strategy has better harmonic and exponential bounds in terms of the total curvature.  Finally, we illustrate our results by considering a task-assignment problem.
\end{abstract}
\section{Introduction}

A variety of combinatorial optimization problems such as generalized assignment (see, e.g., \cite{streeter2008online} and \cite{{FeigeVondrak}}), max $k$-cover (see, e.g., \cite{K-cover1998} and \cite{Feige1998}), maximum coverage location  (see, e.g., \cite{Fisher1977} and \cite{Location}), and sensor placement (see, e.g., \cite{LiC12} and \cite{SensorPlacement}) can be formulated in the following way:

\begin{align}\label{eqn:1}
\begin{array}{l}
\text{maximize} \ \    f(M) \\
\text{subject to} \ \ M\in \mathcal{I}
\end{array}
\end{align}
where $\mathcal{I}$ is a non-empty collection of subsets of a finite set $X$, and $f$ is a real-valued set function defined on the power set $2^X$ of $X$. The set function $f$ is said to be \emph{submodular} if it has the diminishing-return property \cite{Edmonds}. The pair $(X,\mathcal{I})$ is called a \emph{matroid} if the collection $\mathcal{I}$ is hereditary and has the augmentation property \cite{Tutte}. When $\mathcal{I}=\{S\subseteq \mathcal{I}: |S|\leq K\}$ for a given $K$, the pair $(X,\mathcal{I})$ is said to be a \emph{uniform matroid} of rank $K$, where $|S|$ denotes the cardinality of the set $S$. These definitions will be discussed in more detail in Section~II.

 Finding the optimal solution to problem (\ref{eqn:1}) in general is NP-hard. The $1$-batch greedy strategy provides a computationally feasible solution, which starts with the empty set, and then adds one element to the current solution set with the largest gain in the objective function while satisfying the constraints. This scheme is a special case of the \emph{$k$-batch greedy strategy} (with
$k\geq 1$), which starts with the empty set but adds to the current
solution set $k$ elements with the largest gain in the objective
function under the constraints. The performance of the $1$-batch greedy strategy in optimization problems has been extensively investigated, while the performance of the $k$-batch greedy strategy for general $k$ has received little attention, notable exceptions being Nemhauser et al. \cite{nemhauser19781} and Hausmann et al. \cite{hausmann1980}, which we will review in the following subsection.

\subsection{Review of Previous Work}

Nemhauser et al.  \cite{nemhauser19781}, \cite{nemhauser1978} proved that when $f$ is a nondecreasing submodular set function satisfying $f(\emptyset)=0$, the $1$-batch greedy strategy yields at least a $1/2$-approximation for a general matroid and a $(1-1/e)$-approximation for a uniform matroid. By introducing the total curvature $\alpha$,
Conforti and Cornu{\'e}jols   \cite{conforti1984submodular} showed that when $f$ is a nondecreasing submodular set function, the $1$-batch greedy strategy achieves at least a $1/(1+\alpha)$-approximation for a general matroid and a $(1-e^{-\alpha})/{\alpha}$-approximation for a uniform matroid, where the total curvature $\alpha$ is defined as
$$\alpha=\max\limits_{j\in X^*}\left\{1-\frac{f(X)-f(X\setminus\{j\})}{f(\{j\})-f(\emptyset)}\right\}$$ and $X^*=\{j\in X: f(\{j\})>0\}$.
For a nondecreasing submodular set function $f$,  the total curvature $\alpha$ takes values on the interval $ [0,1]$. In this case, we have $1/(1+\alpha)\geq1/2$ and $(1-e^{-\alpha})/\alpha\geq (1-1/e)$, which implies the bounds $1/(1+\alpha)$ and $(1-e^{-\alpha})/\alpha$ are stronger than
the bounds $1/2$ and $(1-1/e)$ in \cite{nemhauser1978} and \cite{nemhauser19781}, respectively. Vondr{\'a}k \cite{vondrak2010submodularity} proved that when $f$ is a nondecreasing submodular set function, the continuous $1$-batch greedy strategy gives at least a $(1-e^{-\alpha})/\alpha$-approximation for any matroid.
 
 Nemhauser et al. \cite{nemhauser19781} proved that when $(X,\mathcal{I})$ is a uniform matroid and $K=ks-p$ ($s$ and $p$ are integers and $0\leq p\leq k-1$), the $k$-batch greedy strategy achieves at least a $(1-(1-\lambda/s)(1-1/s)^{s-1})$-approximation, where  $\lambda=1-p/k$. Hausmann et al. \cite{hausmann1980} showed that when $(X,\mathcal{I})$ is an independence system, then the $k$-batch greedy strategy achieves at least a $q(X,\mathcal{I})$-approximation, where $q(X,\mathcal{I})$ is the rank quotient defined in \cite{hausmann1980}. 
 Although Nemhauser et al. \cite{nemhauser19781} and Hausmann et al. \cite{hausmann1980} investigated the performance of the $k$-batch  greedy strategy, they only considered uniform matroid constraints and independence system constraints, respectively. 
This prompts us to investigate the performance of the $k$-batch greedy strategy more comprehensively.

\subsection{Main Results and Contribution}

In this paper, by defining the total curvature $\alpha_k$ of the objective function, we derive bounds for the performance of the $k$-batch  greedy strategy for a general matroid and a uniform matroid, respectively. 
By comparing the values of $\alpha_k$ for different $k$ and investigating the monotoneity of the bounds, we can compare the performance for different $k$-batch greedy strategies.

The remainder of the paper is organized as follows. In Section~II, we  review the harmonic  and exponential bounds in terms of the total curvature $\alpha$  from \cite{conforti1984submodular}
for a general matroid and a uniform matroid, respectively. In Section~III, we introduce the total curvature $\alpha_k$,
and prove that when $f$ is a nondecreasing submodular set function, the $k$-batch greedy strategy achieves a $1/(1+\alpha_k)$-approximation for a general matroid constraint and a $\left(1-(1-{\alpha}_k/{t})^t\right)/{\alpha}_k$-approximation  for a uniform matroid constraint, where $k$ divides the cardinality of the maximal set in the general matroid, $t=K/k$ is an integer, and $K$ is the rank of the uniform matroid. We also prove that $\alpha_{k}\leq \alpha_{k_1}$ when $f$ is a nondecreasing submodular set function and $k_1$ divides $k$, which implies that the $k$-batch greedy strategy provides  tighter harmonic and exponential bounds compared to the $k_1$-batch greedy strategy.   In Section~IV, we present an application to demonstrate our conclusions. In Section~V, we provide a summary of our work and main contribution.


\section{Preliminaries}\label{sc:II}

In this section, we first introduce some definitions related to sets and curvature. We then review the harmonic  and exponential bounds in terms of the total curvature $\alpha$ from \cite{conforti1984submodular}.
\subsection{Sets and Curvature}

Let $X$ be a finite set, and $\mathcal{I}$ be a non-empty collection of subsets of  $X$. The pair $(X,\mathcal{I})$ is called a \emph{matroid} if 
\begin{itemize}
\item  [i.] For all $B\in\mathcal{I}$, any set $A\subseteq B$ is also in $\mathcal{I}$.
\item  [ii.] For any $A,B\in \mathcal{I}$, if the cardinality of $B$ is greater than that of $A$, then there exists $j\in B\setminus A$ such that $A\cup\{j\}\in\mathcal{I}$.
\end{itemize}

The collection $\mathcal{I}$ is said to be \emph{hereditary} and has the \emph{augmentation} property if it satisfies properties~i and ii,  respectively. The pair $(X,\mathcal{I})$ is called a \emph{uniform matroid} when $\mathcal{I}=\{S\subseteq \mathcal{I}: |S|\leq K\}$ for a given $K$, called the \emph{rank}.

Let $2^X$ denote the power set of $X$, and define the set function $f$: $2^X\rightarrow \mathbb{R^+}$.
The set function $f$ is said to be \emph{nondecreasing} and \emph{submodular} if it satisfies properties~1 and 2 below, respectively:
\begin{itemize}
\item [1.] For any $A\subseteq B\subseteq X$, $f(A)\leq f(B)$.
\item [2.] For any $A\subseteq B\subseteq X$ and $j\in X\setminus B$, $f(A\cup\{j\})-f(A)\geq f(B\cup\{j\})-f(B)$.
\end{itemize}

Property~2 means that the additional value accruing from an extra action decreases as the size of the input set increases, and is also called the \emph{diminishing-return} property in economics.
Property~2 implies that for any $A\subseteq B\subseteq X$ and $T\subseteq X\setminus B$, 
\begin{equation}
\label{eqn:submodularimplies}
f(A\cup T)-f(A)\geq f(B\cup T)-f(B).
\end{equation}
For convenience, we denote the incremental value of adding set $T$ to the set $A\subseteq X$ as $\varrho_T(A)=f(A\cup T)-f(A)$ (following the notation of \cite{conforti1984submodular}).

The \emph{total curvature} of a set function $f$ is defined as \cite{conforti1984submodular}
$$\alpha=\max_{j\in X^*}\left\{1-\frac{\varrho_j({X\setminus\{j\}})}{\varrho_j(\emptyset)}\right\}$$
where $X^*=\{j\in X: \varrho_j(\emptyset)>0\}$. Note that $0\leq \alpha\leq 1$ when $f$ is nondecreasing and submodular, and $\alpha=0$ if and only if $f$ is additive, i.e., $f(X)=f(X\setminus\{j\})+f(\{j\})$ for all $j\in X^*$.

\subsection{Harmonic and Exponential Bounds in Terms of the Total Curvature}
In this section, we review the theorems from \cite{conforti1984submodular} bounding the performance of the $1$-batch greedy strategy using the total curvature $\alpha$ for  general matroid constraints and  uniform matroid constraints.
\begin{Theorem}
\label{Theorem2.1}
Assume that $(X,\mathcal{I})$ is a matroid and $f$ is a nondecreasing submodular set function with $f(\emptyset)=0$ and total curvature $\alpha$. Then the $1$-batch greedy solution $G$ satisfies
$$f(G)\geq \frac{1}{1+\alpha}f(O),$$
where $O$ is the optimal solution of problem (\ref{eqn:1}).
\end{Theorem}

 When $f$ is a nondecreasing submodular set function, we have $\alpha\in[0,1]$, so $1/(1+\alpha)\in[1/2,1]$. Theorem \ref{Theorem2.1} applies to any matroid, which means the bound ${1}/(1+\alpha)$ holds for a uniform matroid  too. Theorem \ref{Theorem2.2} will present a tighter bound when $(X,\mathcal{I})$ is a uniform matroid.

\begin{Theorem}
\label{Theorem2.2}
 Assume that $(X,\mathcal{I})$ is a uniform matroid and $f$ is a nondecreasing submodular set function with $f(\emptyset)=0$ and total curvature $\alpha$. Then the $1$-batch greedy solution $G_K$ satisfies
\begin{align*}
f(G_K)&\geq\frac{1}{\alpha}\left(1-(1-{\alpha}/{K})^K\right)f(O_K)\\
&\geq \frac{1}{\alpha}(1-e^{-\alpha})f(O_K).
\end{align*}
\end{Theorem}
\vspace{2mm}
The function $(1-e^{-\alpha})/\alpha$ is a nonincreasing function of $\alpha$, so $(1-e^{-\alpha})/\alpha\in[1-e^{-1},1]$ when $f$ is a nondecreasing submodular set function. Also it is easy to check $(1-e^{-\alpha})/{\alpha}\geq 1/(1+\alpha)$ for $\alpha\in[0,1]$, which implies that the bound $(1-e^{-\alpha})/{\alpha}$ is stronger than the bound $1/(1+\alpha)$ in Theorem \ref{Theorem2.1}.

\section{Main Results}\label{sc:III}

In this section, first we  define the $k$-batch greedy strategy and the corresponding curvatures that will be used for deriving the harmonic and exponential bounds. Then we derive the performance bounds of the $k$-batch greedy strategy in terms of $\alpha_k$ for general matroid constraints and uniform matroid constraints, respectively. Moreover, we  compare the performance bounds for different $k$-batch greedy strategies.
\subsection{Strategy Formulation and Curvatures}
When $(X,\mathcal{I})$ is a general matroid, assume that the cardinality $K$ of the the maximal set in $\mathcal{I}$ is such that $k$ divides $K$. The $k$-batch greedy strategy is as follows:

Step 1: Let $S^0=\emptyset$ and $t=0$.

Step 2: Select $J_{t+1}\subseteq X\setminus S^t$ for which $|J_{t+1}|=k$, $S^t\cup J_{t+1}\in\mathcal{I}$, and 
\begin{align*}
f(S^t\cup J_{t+1})=\max\limits_{J\subseteq X\setminus S^t\ \text{and}\ |J|=k }f(S^t\cup J),
\end{align*}
 then set $S^{t+1}=S^t\cup J_{t+1}$.

Step 3:  If $f(S^{t+1})-f(S^t)>0$, set $t=t+1$, repeat step~2; otherwise, stop.

When $(X,\mathcal{I})$ is a uniform matroid with rank $K$, without loss of generality, assume that $k$ divides $K$. Then the $k$-batch greedy strategy is as follows:

Step 1: Let $S^0=\emptyset$ and $t=0$.

Step 2: Select $J_{t+1}\subseteq X\setminus S^t$ for which $|J_{t+1}|=k$, and 
\begin{align*}
f(S^t\cup J_{t+1})=\max\limits_{J\subseteq X\setminus S^t\ \text{and}\ |J|=k }f(S^t\cup J),
\end{align*}
 then set $S^{t+1}=S^t\cup J_{t+1}$.

Step 3: If  $t+1<K/k$, set $t=t+1$ and repeat step~2; otherwise, stop. 

Similar to the definition of the total curvature $\alpha$ in \cite{conforti1984submodular}, we define the total curvature $\alpha_k$ for a given $k$ as
$$\alpha_k=\max\limits_{J\in \hat{X}}\left\{1-\frac{\varrho_J(X\setminus J)}{\varrho_J(\emptyset)}\right\}$$ where $\hat{X}=\{J\subseteq X: f(J)>0 \ \text{and}\ |J|=k\}$.

Consider a set $T\subseteq X$ and an ordered set $S=\bigcup_{i=1}^tJ_i\subseteq X$, where $J_i\subseteq X$ and $|J_i|=k$. 
We define $S^0=\emptyset$, $S^i=\bigcup_{l=1}^iJ_l$ for $1\leq i\leq t$, and  the curvature
 \[\bar{\alpha}_k=\max\limits_{i:J_i\subseteq S^*}\left\{\frac{\varrho_{J_i}(S^{i-1})-\varrho_{J_i}(S^{i-1}\cup T)}{\varrho_{J_i}(S^{i-1})}\right\}，\]
 where $S^*=\{J_i\subseteq S-T: |J_i|=k \ \text{and} \ \varrho_{J_i}(S^{i-1})>0\}.$ It is easy to check that $f(S)=\sum_{i=1}^t\varrho_{J_i}(S^{i-1})$ and $\bar{\alpha}_k\leq \alpha_k$. 

For a uniform matroid with rank $K$, we use $S_K=\bigcup_{i=1}^tJ_i$ to denote the $k$-batch greedy solution, where $J_i$ is the set selected by the $k$-batch greedy strategy at stage $i$. Assume that $O_K$ is the optimal solution to Problem 1.
We define the curvature $\hat{\alpha}_k$ with respect to the optimal solution  as 
\[\hat{\alpha}_k=\max\limits_{1\leq j \leq t}\left\{1-\frac{\varrho_{S^j}(O_K)}{\varrho_{S^j}(\emptyset)}\right\}.\]
It is easy to prove that $\hat{\alpha}_k\leq \alpha_k$ when $f$ is a nondecreasing submodular set function.

\subsection{Harmonic Bound and Exponential Bound in Terms of the Total Curvature}
 
 The following proposition will be applied to derive the performance bounds for both general matroid constraints and  uniform matroid constraints.
\begin{Proposition}
\label{Pro1}
If $f$ is a nondecreasing submodular set function on $X$, $S$ and $T$ are subsets of $X$, and $\{T_1,\ldots, T_r\}$ is a partition of $T\setminus S$, then 
\begin{equation}
\label{eqn:Prop1}
f(T\cup S)\leq f(S)+\sum\limits_{i:T_i\subseteq T\setminus S}\varrho_{T_i}(S).
\end{equation}
\end{Proposition}
\vspace{1mm}
\begin{proof}
By the assumption that $\{T_1,\ldots, T_r\}$ is a partition of $T\setminus S$ and inequality \ref{eqn:submodularimplies},  we have 
 \begin{align*}
 f(T\cup S)-f(S)&=f(S\cup \bigcup_{l=1}^r T_l)-f(S)\\
 &=\sum\limits_{j=1}^r \varrho_{T_j}(S\cup\bigcup_{l=1}^{j-1}T_l)\\
 &\leq \sum\limits_{j:T_j\subseteq T\setminus S}\varrho_{T_j}(S).
 \end{align*}
\end{proof}
 The following proposition will be applied to derive the performance bound for  general matroid constraints.
 \begin{Proposition}
 \label{Pro2}
 Assume that $f$ is a nondecreasing submodular set function on $X$ with $f(\emptyset)=0$. Given a set $T\subseteq X$, a partition $\{T_1,\ldots, T_r\}$  of $T\setminus S$, and an ordered set $S=\bigcup_{i=1}^tJ_i\subseteq X$ with $|J_i|=k$, we have 
 \begin{align}
 \label{ineq:prop2}
  f(T)\leq \bar{\alpha}_k\sum\limits_{i:J_i\subseteq S\setminus T}&\varrho_{J_i}(S^{i-1})+\sum\limits_{i:J_i\subseteq T\cap S}\varrho_{J_i}(S^{i-1})\nonumber\\
  &+\sum\limits_{i:T_i\subseteq T\setminus S}\varrho_{T_i}(S).
 \end{align}
 \end{Proposition}
 \begin{proof}
By the definition of the curvature $\bar{\alpha}_k$, we have
 \begin{align*}
 f(T\cup S)-f(T)&=\sum\limits_{i=1}^t\varrho_{J_i}(T\cup S^{i-1})\\
 &=\sum\limits_{i:J_i\subseteq S\setminus T}\varrho_{J_i}(T\cup S^{i-1})\\
 &\geq (1-\bar{\alpha}_k)\sum\limits_{i:J_i\subseteq S\setminus T}\varrho_{J_i}( S^{i-1}).
 \end{align*}
 
 By Proposition \ref{Pro1}, we have 
 \[f(T\cup S)\leq f(S)+\sum\limits_{i:T_i\subseteq T\setminus S}\varrho_{T_i}(S).\]
 
 Combining the inequalities above and using the identity \[f(S)=\sum\limits_{i:J_i\subseteq S\setminus T}\varrho_{J_i}( S^{i-1})+\sum\limits_{i:J_i\subseteq T\cap S}\varrho_{J_i}(S^{i-1}),\] we  get the inequality (\ref{ineq:prop2}).
 \end{proof}

Recall that when $(X,\mathcal{I})$ is a general matroid, we assume that $k$ divides  the cardinality $K$ of the maximal set in $\mathcal{I}$. By the augmentation property of a general matroid, any greedy solution and optimal solution can be augmented to a set of length $K$, respectively. Let $S=\bigcup_{i=1}^tJ_i$ be the $k$-batch greedy solution, where $J_i$ is the set selected by the $k$-batch greedy strategy at the $i$th step for $1\leq i\leq t$. Let $O=\{o_1,\ldots, o_K\}$ be the optimal solution. We prove that the following lemma holds.

\begin{Lemma}
\label{lemma1}
The optimal solution $O=\{o_1,\ldots, o_K\}$ can be ordered as $O=\bigcup_{i=1}^tJ_i'$  such that $\varrho_{J_i'}(S^{i-1})\leq \varrho_{J_i}(S^{i-1})$, where ${J_1',\ldots,J_t'}$ is a partition of $O$ and $|J_i'|=k$ for $1\leq i\leq t$. Furthermore, if $J_i'\subseteq O\cap S$, then $J_i'=J_i$.
\end{Lemma}
\begin{proof}
Similar to the proof in \cite{nemhauser19781}, we will prove this lemma by backward induction on $i$ for $i=t, t-1,\ldots, 1$. Assume that $J_l'$ satisfies the inequality $\varrho_{J_l'}(S^{l-1})\leq \varrho_{J_l}(S^{l-1})$ for $l>i$, and let $O^i=O\setminus \bigcup_{l>i} J_l'$. Consider the sets $S^{i-1}$ and $O^i$. By definition, $|S^{i-1}|=(i-1) k$ and $|O^i|=i k$.  Using the augmentation property of a general matroid, we have that there exists one element $o_{i_1}\in O^i\setminus S^{i-1}$ such that $S^{i-1}\cup\{o_{i_1}\}\in\mathcal{I}$. Next consider $S^{i-1}\cup\{o_{i_1}\}$ and $O^i$. Using the augmentation property again, there exists one element $o_{i_2}\in O^i\setminus S^{i-1}\setminus\{o_{i_1}\}$ such that $S^{i-1}\cup\{o_{i_1}, o_{i_2}\}\in\mathcal{I}$. Similar to the process above, using the augmentation property $(k-2)$ more times, finally we have that there exists $J_i'=\{o_{i_1},\ldots,o_{i_k}\}\subseteq O^i\setminus S^{i-1}$ such that $S^{i-1}\cup J_i'\in \mathcal{I}$. By the $k$-batch greedy strategy, we have that $\varrho_{J_i'}(S^{i-1})\leq \varrho_{J_i}(S^{i-1})$. Furthermore, if $J_i\subseteq O^i$, we can set $J_i'=J_i$. 
\end{proof}

The following two theorems  present our performance bounds in terms of the total curvature $\alpha_k$ for the $k$-batch greedy strategy under a general matroid constraint and a uniform matroid, respectively.
\begin{Theorem}
\label{Theorem3.3}
Assume that  $f$ is a nondecreasing submodular set function with $f(\emptyset)=0$, the pair $(X,\mathcal{I})$ is a general matroid, and $k$ divides the cardinality $K$ of the maximal set in $\mathcal{I}$. Then the $k$-batch greedy strategy $S=\bigcup_{i=1}^tJ_i$ satisfies 
\begin{equation}
\label{ineq:generalbound}
f(S)\geq \frac{1}{1+\alpha_k}f(O).
\end{equation}
\end{Theorem}
\vspace{2mm}
\begin{proof}
By Lemma~\ref{lemma1}, we have that the optimal solution $O$ can be ordered as $O=\bigcup_{i=1}^tJ_i'$ such that  $\varrho_{J_i'}(S^{i-1})\leq \varrho_{J_i}(S^{i-1})$, where $\{J_l'\}_{l=1}^t$ is a partition of $O$ and $|J_l'|=k$ for $1\leq l\leq t$.

By Proposition \ref{Pro2}, we have 
\begin{align*}
f(O)\leq \bar{\alpha}_k\sum\limits_{i:J_i\subseteq S\setminus O}&\varrho_{J_i}(S^{i-1})+\sum\limits_{i:J_i\subseteq O\cap S}\varrho_{J_i}(S^{i-1})\\
  &+\sum\limits_{i:J_i'\subseteq O\setminus S}\varrho_{J_i'}(S).
\end{align*}

By inequality (\ref{eqn:submodularimplies}), we have 
\[\varrho_{J_i'}(S)\leq \varrho_{J_i'}(S^{i-1})\leq \varrho_{J_i}(S^{i-1}).\]
Then 
\begin{align*}
f(O)&\leq \bar{\alpha}_k\sum\limits_{i:J_i\subseteq S\setminus O}\varrho_{J_i}(S^{i-1})+\sum\limits_{i:J_i\subseteq O\cap S}\varrho_{J_i}(S^{i-1})\\
&\quad\quad\quad+\sum\limits_{i:J_i'\subseteq O\setminus S}\varrho_{J_i}(S^{i-1})\\
&\leq {\alpha}_kf(S)+f(S),
\end{align*}
which implies that $f(S)\geq \frac{1}{1+\alpha_k}f(O)$.
\end{proof}

\emph{Remarks}
\begin{itemize}
\item The harmonic bound $1/(1+\alpha_k)$ for the $k$-batch greedy strategy holds for \emph{any} matroid. However, for uniform matroids, a better bound is given  in Theorem \ref{Theorem3.4}.
\item The function $g(x)=1/(1+x)$ is nonincreasing in $x$ on the interval $[0,1]$. 
\end{itemize}

\begin{Theorem}
\label{Theorem3.4}
Assume that $f$ is a nondecreasing submodular set function with $f(\emptyset)=0$, the pair $(X,\mathcal{I})$ is a uniform matroid with rank $K$,  and $k$ divides $K$. Then the $k$-batch greedy solution $S_K=\bigcup_{i=1}^tJ_i$ satisfies 
\begin{align}
\label{k-batchuniformbound}
f(S_K)&\geq \frac{1}{{\alpha}_k}\left(1-(1-\frac{{\alpha}_k}{t})^t\right)f(O_K)\nonumber\\
&\geq\frac{1}{{\alpha}_k}(1-e^{-\alpha_k})f(O_K).
\end{align}

\end{Theorem}
\vspace{2mm}
\begin{proof}
Taking $T$ to be the optimal solution $O_K$ and $S$ to be the set $S^j$ generated by the $k$-batch greedy strategy over the first $j$ stages in Proposition~\ref{Pro1} results in 
\[f(O_K\cup S^j)\leq f(S^j)+\sum\limits_{i:T_i\subseteq O_K\setminus S^j}\varrho_{T_i}(S^j),\] 
where $|T_i|=k$.

By the $k$-batch greedy strategy, we have that for $T_i\subseteq O_K\setminus S^j$, $$\varrho_{T_i}(S^j)\leq \varrho_{J_{j+1}}(S^j),$$
which implies that 
\begin{equation}
\label{ineq:relation}
f(O_K\cup S^j)\leq f(S^j)+t\varrho_{J_{j+1}}(S^j).
\end{equation}
By the definition of $\hat{\alpha}_k$, we have 

\[f(O_K)+(1-\hat{\alpha}_k)f(S^j)\leq f(O_K\cup S^j).\]
Combining the inequality above and (\ref{ineq:relation}), we have 
\begin{equation}
\label{ineq:iteration}
f(S^{j+1})\geq \frac{1}{t}f(O_K)+(1-\frac{\hat{\alpha}_k}{t})f(S^j).
\end{equation}
Taking $j=0,1,\ldots, t-1$ in (\ref{ineq:iteration}), we have 
\begin{align*}
f(S_K)=f(S^t)&\geq \frac{1}{t}f(O_K)+(1-\frac{\hat{\alpha}_k}{t})f(S^{t-1})\\
&\geq \frac{1}{t}f(O_K)\sum\limits_{l=0}^{t-1}(1-\frac{\hat{\alpha}_k}{t})\\
&=\frac{1}{\hat{\alpha}_k}\left(1-(1-\frac{\hat{\alpha}_k}{t})^t\right)f(O_K),
\end{align*}
which implies
\begin{align*}
f(S_K)&\geq \frac{1}{{\alpha}_k}\left(1-(1-\frac{{\alpha}_k}{t})^t\right)f(O_K)\\
&\geq\frac{1}{{\alpha}_k}(1-e^{-\alpha_k})f(O_K).
\end{align*}
\end{proof}

\textit{Remarks}

\begin{itemize}
\item When $\alpha_k=1$, the bound $(1-(1-\alpha_k/t)^t)/\alpha_k$ becomes $1-(1-1/t)^t$, which is the bound  in \cite{nemhauser19781} when $p=0$.
\item Let $h(x,y)=\left(1-(1-{x}/{y}\right)^y)/{x}$. The function $h(x,y)$ is nonincreasing in $x$ on the interval $[0,1]$ for any positive integer $y$.
Also $h(x,y)$ is nonincreasing in $y$ when $x$ is a constant on the interval $[0,1]$.
\item The function $l(x)=(1-e^{-x})/{x}$ is nonincreasing in $x$, so $(1-e^{-\alpha_k})/{\alpha_k}\in[1-e^{-1}, 1]$.
\item The monotoneiety of $g(x)$ and $h(x,y)$ implies that the $k$-batch greedy strategy has  better harmonic and exponential bounds than the $1$-batch greedy strategy if $\alpha_k\leq \alpha$ .
\end{itemize}

The following theorem establishes that indeed $\alpha_k\leq \alpha$.

\begin{Theorem}
\label{Theorem3.5}
Assume that $f$ is a nondecreasing submodular set function satisfying $f(\emptyset)=0$. Then $\alpha_k\leq \alpha.$ 
\end{Theorem}
\begin{proof}
By the definition of $\alpha_k$, we have 
\begin{align*}
\alpha_k&=\max_{J_k\subseteq \hat{X}}\left\{1-\frac{\varrho_{J_k}({X\setminus J_k})}{\varrho_{J_k}(\emptyset)}\right\}\\
&=1-\min_{J_k\subseteq \hat{X}}\left\{\frac{\sum\limits_{l=1}^k\varrho_{j_l}(X\setminus J_l)}{\sum\limits_{l=1}^k\varrho_{j_l}(J_{l-1})}\right\},
\end{align*}
where  $J_l=\{j_1,\ldots, j_l\}$ for $1\leq l\leq k$.

By the assumption that $f$ is a submodular set function, we have, for $1\leq l\leq k$,
$$\varrho_{j_l}(X\setminus J_l)\geq \varrho_{j_l}(X\setminus\{j_l\})
\ \text{and}\ \varrho_{j_l}(J_{l-1})\leq \varrho_{j_l}(\emptyset),$$
which imply that 
\[\frac{\sum\limits_{l=1}^k\varrho_{j_l}(X\setminus J_l)}{\sum\limits_{l=1}^k\varrho_{j_l}(J_{l-1})}\geq\frac{\sum\limits_{l=1}^k\varrho_{j_l}(X\setminus\{j_l\})}{\sum\limits_{l=1}^k\varrho_{j_l}(\emptyset)}. \]
Then, we have
\begin{equation}
\label{Inequality2}
\alpha_k\leq 1-\min_{{j_1,\ldots,j_k}\in \hat{X}}\left\{\frac{\sum\limits_{l=1}^k\varrho_{j_l}(X\setminus\{j_l\})}{\sum\limits_{l=1}^k\varrho_{j_l}(\emptyset)}\right\}.
\end{equation}

By the definition of $\alpha$, we have for $1\leq l\leq k$,
\[\varrho_{j_l}(X\setminus\{j_l\})\geq (1-\alpha)\varrho_{j_l}(\emptyset).\]

Combining the inequality above and (\ref{Inequality2}), we have 
\[\alpha_k\leq 1-(1-\alpha)=\alpha.\]

\end{proof}

The following theorem states that if $k_1$ divides $k$, then  the total curvature $\alpha_{k}$ for the $k$-batch greedy is smaller than the total curvature $\alpha_{k_1}$ for the $k_1$-batch greedy strategy.

\begin{Theorem}
\label{Theorem3.6}
Assume that $f$ is a submodular set function satisfying $f(\emptyset)=0$. Then $\alpha_{k}\leq \alpha_{k_1}$ when $k_1$ divides $k$.
\end{Theorem}
\begin{proof}
Suppose that $k=k_1k_2$ ($k_1$ and $k_2$ are integers).
Write
\begin{align*}
\varrho_{J_k}&(X\setminus J_k)=\sum\limits_{l=1}^{k_2}\varrho_{J_{lk_1}\setminus J_{(l-1)k_1}}(X\setminus J_{l {k_1}})
\end{align*}
and 
$$\varrho_{J_k}(\emptyset)=
\sum\limits_{l=1}^{k_2}\varrho_{J_{l k_1}\setminus J_{(l-1) k_1}}(J_{(l-1) k_1}).$$

By inequality (\ref{eqn:submodularimplies}), we have for $1\leq l\leq k_2$, 
\begin{align*}
&\varrho_{J_{l k_1}\setminus J_{(l-1) k_1}}(X\setminus J_{l k_1})\geq\\
&\varrho_{J_{l k_1}\setminus J_{(l-1) k_1}}(X\setminus (J_{l k_1}\setminus J_{(l-1) k_1}))
\end{align*}
and\[\varrho_{J_{l k_1}\setminus J_{(l-1) k_1}}(J_{(l-1) k_1})
 \leq \varrho_{J_{l k_1}\setminus J_{(l-1) k_1}}(\emptyset).\]

From the inequalities above and by the definition of $\alpha_k$, we have 

\begin{align*}
\alpha_k&=\max_{J_k\subseteq \hat{X}}\left\{1-\frac{\varrho_{J_k}({X\setminus J_k})}{\varrho_{J_k}(\emptyset)}\right\}\\
&=1-\min_{J_k\subseteq \hat{X}}\left\{\frac{\sum\limits_{l=1}^{k_2}\varrho_{J_{l k_1}\setminus J_{(l-1) k_1}}(X\setminus J_{l k_1})}{\sum\limits_{l=1}^{k_2}\varrho_{J_{l k_1}\setminus J_{(l-1) k_1}}(J_{(l-1) k_1})}\right\}\\
&\leq 1-\min_{J_k\subseteq\hat{X}}\\
&\tiny{\left\{\frac{\sum\limits_{l=1}^{k_2}\varrho_{J_{l k_1}\setminus J_{(l-1) k_1}}(X\setminus (J_{l k_1}\setminus J_{(l-1) k_1}))}{\sum\limits_{l=1}^{k_2}\varrho_{J_{l k_1}\setminus J_{(l-1) k_1}}(\emptyset)}
\right\}}.
\end{align*}

By the definition of $\alpha_{k_1}$, we have for $1\leq l\leq k_2$, 
\begin{align*}
\varrho_{J_{l k_1}\setminus J_{(l-1) k_1}}(X\setminus (J_{l k_1}\setminus J_{(l-1) k_1}))\\
\geq (1-\alpha_{k_1})\varrho_{J_{l k_1}\setminus J_{(l-1) k_1}}(\emptyset).
\end{align*}

Using the inequalities above, we have 
$$\alpha_k\leq 1-(1-\alpha_{k_1})=\alpha_{k_1}.$$
\end{proof}

One would also expect the following generalization of Theorem~\ref{Theorem3.6} to hold: if $k_1\leq k$, then $\alpha_k\leq\alpha_{k_1}$, leading to better bounds for the $k$-batch greedy strategy than for the $k_1$-batch greedy strategy. We have a proof for this claim using Lemmas~1.1 and 1.2 in \cite{vondrak2010}, but the proof is more involved and is omitted for the sake of brevity. We will illustrate the validity of this claim in Section~IV.

\section{Application: Task Assignment}

In this section, we consider a task assignment problem to demonstrate that the $k$-batch greedy strategy has better performance than the $k_1$-batch greedy strategy when $f$ is a nondecreasing submodular set function.

As a canonical example for problem (\ref{eqn:1}), we  consider the task assignment problem  posed in \cite{streeter2008online}, which was also analyzed in \cite{ZhC13J} and \cite{YJ2015}.  In this problem, there are $n$ subtasks and a  set $X$ of $N$ agents $a_j$ $(j=1,\ldots, N).$ At each stage, a subtask $i$ is assigned to an agent $a_j$, who accomplishes the task with probability $p_i(a_j)$. Let $X_i({a_1,a_2,\ldots, a_k})$ denote the random variable that describes whether or not subtask $i$ has been accomplished after performing the sequence of actions ${a_1,a_2,\ldots, a_k}$ over $k$ stages. Then $\frac{1}{n}\sum_{i=1}^n	X_i(a_1,a_2,\ldots,a_k)$ is the fraction of subtasks accomplished after $k$ stages by employing agents $a_1,a_2,\ldots, a_k$. The objective function $f$ for this problem is the expected value of this fraction, which can be written as

$$f(\{a_1,\ldots,a_k\})=\frac{1}{n}\sum_{i=1}^n\left(1-\prod_{j=1}^k(1-p_i(a_j))\right).$$

Assume that $p_i(a)>0$ for any $a\in X$. Then it is easy to check that $f$ is nondecreasing. Therefore, when $\mathcal{I}=\{S\subseteq X: |S|\leq K\}$, the solution to this problem should be of length $K$.  Also, it is easy to check that  $f$ has the diminishing-return property.

For convenience, we only consider the special case $n=1$; our analysis can be generalized to any $n\geq 2$. For $n=1$, we have 
$$f(\{a_1,\ldots,a_k\})=1-\prod_{j=1}^k(1-p(a_j))$$
where $p(\cdot)=p_1(\cdot)$.

Assume that $0<p(a_1)\leq p(a_2)\leq \cdots\leq p(a_N)\leq 1$. Then by the definition of the total curvature $\alpha_k$, we have 
\begin{align*}
\alpha_k&=\max\limits_{j_1,\ldots,j_k\in {X}}\left\{1-\frac{f(X)-f(X\setminus\{j_1,\ldots,j_k\})}{f(\{j_1,\ldots, j_k\})-f(\emptyset)}\right\}\\
&=1-\prod_{l=k+1}^K(1-p(a_l)).
\end{align*}

From the form of $\alpha_k$, we have $\alpha_k\in[0,1]$, which is consistent with our conclusion that when $f$ is a nondecreasing submodular set function, then $\alpha_k \in[0,1]$. Also we have $\alpha_k\leq \alpha_{k_1}$  when $k_1$ divides $k$. Even if $k_1$ does not divide $k$, we still have $\alpha_k\leq \alpha_{k_1}$ in this example, which is consistent with our claim.
\section{Conclusion}
In this paper, we derived performance bounds for the $k$-batch greedy strategy, $k\geq 1$, in terms of a total curvature $\alpha_k$.
We showed that when the objective function is nondecreasing and submodular, the $k$-batch greedy strategy satisfies a harmonic bound  $1/(1+\alpha_k)$ for a general matroid  and an exponential bound $(1-e^{-\alpha_k})/\alpha_k$ for a uniform matroid, where $k$ divides the cardinality of the maximal set in the general matroid and the rank of the uniform matroid, respectively. We proved that, for a submodular objective function, $\alpha_k\leq \alpha_{k_1}$ when $k_1$ divides $k$. Consequently, for a nondecreasing submodular objective function,  the $k$-batch greedy strategy has better performance bounds than the $k_1$-batch greedy strategy in such a case. This is true even when $k_1\leq k$ does not divide $k$, but it follows a more involved proof that we have left out. We demonstrated our results by considering a task-assignment problem, which also corroborated our claim that if $k_1\leq k$, then $\alpha_{k}\leq \alpha_{k_1}$ even if $k_1$ does not divide $k$.

\bibliographystyle{IEEEbib}

\end{document}